\documentclass[reqno,11pt]{amsart}
\input{packages.tex}

\begin{document}
\title[Dynamical fitness models]{Dynamical fitness models: evidence of universality classes for preferential attachment graphs}

\author[A. Cipriani]{Alessandra Cipriani}
\author[A. Fontanari]{Andrea Fontanari}
\address{TU Delft (DIAM), Building 28, van Mourik Broekmanweg 6, 2628 XE, Delft, The Netherlands}
\email{A.Cipriani@tudelft.nl, A.Fontanari@tudelft.nl}
\date{\today}

\begin{abstract}
In this paper we define a family of preferential attachment models for random graphs with fitness in the following way: independently for each node, at each time step a random fitness is drawn according to the position of a moving average process with positive increments. We will define two regimes in which our graph reproduces some features of two well-known preferential attachment models: the Bianconi--Barab\'asi and the Barab\'asi--Albert models. We will discuss a few conjectures on these models, including the convergence of the degree sequence and the appearance of Bose--Einstein 
condensation in the network when the drift of the fitness process has order comparable to the graph size.
 \end{abstract}

\keywords{Preferential attachment with fitness, Barab\'asi--Albert model, Bianconi--Barab\'asi model, majorization, orderings, Bose--Einstein condensation}
\subjclass[2000]{05C80, 60E15}

\maketitle

\section{Introduction}
Preferential attachment models (PAMs) are a type of dynamic network that exhibits features observed in many real-life datasets, such as the power-law decay in the tails of the degree distribution~\parencite{remco2016random}. Since the works~\cite{yule1925ii},~\cite{simon1955class}, one of the versions of preferential attachment graphs which became prominent is the Barab\'asi--Albert model~\parencite{barabasialbert}. In the simplest case, at each discrete time step a new node attaches itself to one of the already existing vertices with probability proportional to that vertex's degree. One of the main features of this model is the {\em old-get-richer} phenomenon, for which older vertices tend to accumulate higher degree. PAMs have been extended to allow for different attachment probabilities. In physics it is relevant to look at the following generalisation: each node comes into the network with an additional label called {\em fitness}, sampled at random. Now the attachment probability to a node is not only proportional to its degree, but to its degree times its fitness. This graph is called PAM with fitness, first introduced by~\cite{bianconi2001competition}. We shall henceforth refer to this model as the Bianconi--Barab\'asi model. One of the main interests in fitness models of preferential attachment is due to their link to a well-known phenomenon called ``Bose--Einstein condensation''~\parencite{barabasi2016network}. Roughly speaking, condensation for a graph means that a small fraction of the nodes collects a sum of degrees which is linear in the network size. In physical terms this means that particles in a Bose--Einstein gas (corresponding to nodes in our graph) crowd at the lowest energy level (roughly corresponding to the fitness). It has been shown (\cite{borgs2007first,dereich2016, dereich2017nonextensive,dereich_ortgiese_2014} are some of the many references) that under suitable conditions on the fitness distribution a condensate appears in the Bianconi--Barab\'asi model.
In recent years, applications for preferential attachment models with fitness went beyond physics. On the mathematical side, several variants of the Bianconi--Barab\'asi model have been developed to study condensation and related phenomena~\parencite{garavaglia,haslegrave2016preferential,haslegrave2019condensation}. On the modelling side they have been used to study the power-law exhibited by cryptocurrency transaction networks such as Bitcoins~\parencite{BitcoinFitness}, and citation networks~\parencite{Garavaglia2017}. The reason for this is that one can think that nodes represent agents and connections among them depend on their reputation (the degree) and their skills (the fitness). Studying properties of such networks can lead to better understanding of real-life phenomena.

So far the fitness has been considered fixed in time. Clearly this may not be the case, for example when the skills of an agent increase in time via a learning-by-doing mechanism. Motivated by this, we want to define a family of PAMs with {\em dynamical fitness}. In our networks, the fitness $\F_t(v)$ is allowed to vary over time, independently for each vertex $i$, according to a stochastic process that arises out of the sum of i.i.d. positive random variables $\epsilon$. It turns out that by choosing the right summation scheme we are able to range between the Barab\'asi--Albert and the Bianconi--Barab\'asi model. More precisely, we start by showing that if the amount of increments $\epsilon$ is finite, we are able to compare our model to the Barab\'asi--Albert and Bianconi--Barab\'asi models by showing that some features resemble those of these benchmark cases (eg. expected asymptotic attachment probability). Furthermore, we investigate numerically the behavior of the degree sequence and show it is asymptotically close to that of the Barab\'asi--Albert resp. Bianconi--Barab\'asi model. Indeed the {\em old-get-richer} phenomenon is reinforced by the presence of larger fitness of older nodes.  

We continue by focusing on the Bianconi--Barab\'asi model proving that condensation can be induced by summing sufficiently many increments. Provided the mean increment $\mu_\epsilon$ is less than ${1}/{2}$ we can always find a fitness distribution $\nu(\epsilon)$ such that the Bianconi--Barab\'asi model with that fitness condensates, in particular $\nu(\epsilon)$ is obtained via convolution of the increment distribution. However, for a larger mean increment condensation will not appear regardless of the growth of the fitness (as long as it is bounded in time), thus showing a phase transition at $\mu_\eps={1}/{2}.$ We then conjecture, and provide numerics in support, that this behavior carries over to our model as well.

Finally, we conclude with several open problems and conjectures. In particular, we perform simulations suggesting the appearance of a condensate when the sum of the increments grows linearly in the network size. We inquire whether this phase transition is universal in the increment law.  

Drawing conclusions from our work, both mathematical and empirical evidence hints at the universal behavior of the Barab\'asi--Albert and the Bianconi--Barab\'asi model which appear to be stable under random, but bounded in time, perturbations of the attachment probability.

The main challenge at present is that the study of random networks with fitness has been successfully carried through via a coupling with continuous-time branching processes (CTBP) and generalised P\'olya's urns. However, our time-varying fitness corresponds to a change of the reproduction rate in each family of the associated CTBP, {\em at every birth event}, making the CTPB lose its Markovian properties and independence over families. Using the observation that the degree sequence corresponds to an element of a simplex, the theory of majorization turns out to be a good tool to control the quantities we are interested in. To the best of the authors's knowledge, majorization seems to be newly applied in the context of PAMs, although it has been widely used for other random graphs \parencite{arnold2007}.

\paragraph*{Structure of the article}In Section~\ref{sec:model_res} we describe the model we are considering and state our main Theorems, which are proved, together with auxiliary results, in Section~\ref{sec:proofs}. We describe some conjectures and numerical simulations in Section~\ref{sec:conj}. We finally conclude with the remarks of Section~\ref{sec:concl}.
\paragraph*{Acknowledgments}The authors are grateful to Remco van der Hofstad and C\'ecile Mailler for several stimulating discussions. The first author is supported by the grant 613.009.102 of the Netherlands Organisation for Scientific Research (NWO). The second author acknowledges the support of grant FP7 Marie Curie ITN Project WakEUpCall Nr. 643045.
\subsection*{Notation} For a random variable $X$ we denote its expectation as $\mu_X$. We say that $f(x)\asymp g(x)$ if there exist 
universal constants $c_r,\,c_\ell>0$ such that $c_\ell g(x)\le f(x)\le c_r g(x)$ for all $x.$ We use the standard notation for graphs $G_t=(V_t,\,E_t)$ where $V_t$ denotes the vertex set at time $t$ and $E_t$ denotes the edge set. The degree of a node $v$ at time $t$ in a graph $G=G_t$ is indicated as $\deg_t(v)$. We write $w\to v$ to indicate that node $w$ is attached to node $v$ in the graph $G$, and the bond between the two nodes is written as $(w,\,v)$. 
We will use bold fonts for vectors.
For $\bo{x}=(x_1,\,\ldots,\,x_d)\in\R^d$ let
$$x_{[1]} \ge \cdots\ge  x_{[d]}$$
denote the components of $x$ in decreasing order, and let
$$\bo{x}_\downarrow: = (x_{[1]},\, \ldots\,  x_{[d]})$$
denote the decreasing rearrangement of $\bo x$. We will use the Landau notation always with respect to a time parameter $t\to\infty.$ For a probability law $\nu$ we denote as $\nu^{(m)}$ the $m$-convolution of $\nu$. Finally, $d_{TV}$ denotes the total variation distance.
\section{The model and main results}\label{sec:model_res}
\subsection{Definition of the model}\label{subsec:def}
We will now begin by setting up the definition of our preferential attachment graphs. In particular in this work we will construct random trees where at each time step a new node attaches itself to a previous one according to the preferential attachment rule given in Equation~\eqref{eq:attach_rule}.

The construction algorithm is the following:
\begin{algorithm}[ht!]
\caption{Construction of the preferential attachment graph}\label{euclid}
\begin{algorithmic}[1]
\Procedure{GraphConstruction}{}
\BState \emph{initial step}:
\State $G_1=(V_1,\,E_1) \gets (\{1,\,2\},\,(1,2))$
\State $\F_1\gets (\F_1(1),\,\F_1(2))$
\BState \emph{recursive step}:
\For {$i > 1$} 
\State $\F_i\gets (\F_i(v))$ for all $v\in V_{i-1}$.
\State node $\gets j\in V_{i-1}$ chosen with probability  \begin{equation}\label{eq:attach_rule}
        P(i\to j)=\frac{\F_{i}(j)\deg_{i-1}(j)}{\sum_{v\in V_{i-1}}\F_i(v)\deg_{i-1}(v)}
        \end{equation}
\State $V_i\gets V_{i-1}\cup\{i\}$
\State $E_i=E_{i-1}\cup \{(i,\,node)\}$
\EndFor
\EndProcedure
\end{algorithmic}
\end{algorithm}
  
The choice of the initial graph $G_1$ is arbitrary and does not affect our results. The random variable 
    $$Z_t:=\sum_{v\in V_{t-1}}\F_t(v)\deg_{t-1}(v)$$ 
    is the partition function at time $t$. Note also that we label a vertex by its arrival time in the graph. This mapping is valid since our graph is a tree. Therefore we will interchangeably use $u,\,v,\,w,\ldots$ as names of vertices or as times in the graph evolution without risk of confusion.

    The new feature of our model is that fitnesses randomly vary in time according to a specific discrete-time stochastic process whose definition is given in Definition~\ref{def:fit}.
    \begin{defi}[Fitnesses as stochastic processes]\label{def:fit}Given the graph $G_{t-1}=(V_{t-1},\,E_{t-1}) $ we set $(\F_t(v))_{v\in V_{t-1}}$ as
\begin{equation}\label{eq:fitness_process}
\F_t(v):=\sum_{i=v+1}^{t}\alpha_i\epsilon_i(v),
\end{equation}
where $\bm{\alpha}=(\alpha_i)_{i=1}^t$ is such that $\alpha_i\in\{0,\,1\}$ for all $i=1,\,\ldots,\,t$, and $(\epsilon_i(v))_{i\in\N,\,v\in V_{t-1}}$ is a collection of i.i.d. non-negative random variables.\end{defi}
The law of $\epsilon$ is denoted by $\nu$. In the present work, we assume that $\text{supp}(\nu)\subset [0,\,1]$, as common in the literature on condensation for preferential attachment models with fitness (\cite{borgs2007first,dereich2017nonextensive,haslegrave2019condensation}).

According to the choice of $\bm\alpha$,~\eqref{eq:fitness_process} spans a rich variety of stochastic processes. Some notable ones are 
\begin{enumerate}[label=\alph*),ref=\alph*)]
    \item\label{item:a_m} the i.i.d.~sampling from the law $\nu$ if $\bm{\alpha}=(0,\,\ldots,\,0,\,1)$, namely for every $v$ one has $\F_{i}(v)= \epsilon_i(v).$ Observe that $\F_i(v)$ is independent from $\F_j(v)$ for all $i\neq j$. 
    \item\label{item:b_m} A moving average process $MA(m)$ of order $m$, $m<\infty$, for $\displaystyle \bm{\alpha}=(0,\,0,\,\ldots,0,\,\underbrace{1,\,\ldots,\,1}_{m})$. Observe that, for fixed $v$ and $i$, $\F_{i}(v)$ is independent from $\F_{i+m}(v)$ ($m$-Markov property).
    \item\label{item:c_m} The random walk with positive increments $\eps$ for $\bm{\alpha}=(1,\,\ldots,\,1).$
    \item\label{item:d_m} When $\bm{\alpha}=(\underbrace{1,\,\ldots,\,1}_{m},\,0,\,\ldots,\,0)$ with $m<\infty$ the fitness is a finite sum of increments such that $\F_i(v)=\F_{v+1+m}(v)$ for all $i\ge v+1+m$.
    \item\label{item:e_m} When $\bm{\alpha}=(1,\,0,\,\ldots,\,0)$ we obtain a time-independent fitness such that for every $v$ one has $\F_i(v)=\epsilon_{v+1}(v)$ for all $i\ge v+1$.
\end{enumerate}
The classical Bianconi--Barab\'asi model corresponds to~\ref{item:e_m}. On the other hand of the spectrum, we will provide evidence (see Section~\ref{sec:conj}) to support the conjecture that the model in Case~\ref{item:a_m} resembles a Barab\'asi--Albert model. In our opinion this motivates the choice of the summation scheme in~\eqref{eq:fitness_process} since it allows for a possible interpolation between the Bianconi--Barab\'asi and the Barab\'asi--Albert models.

In the next Subsection we will indeed identify two regimes in which the behavior of our model follows closely that of the Bianconi--Barab\'asi resp. Barab\'asi--Albert graphs.

\subsection{Identification of two regimes}
The relevance of our model lies in the fact that, by a suitable tuning of the vector $\bm{\alpha}$, we can construct two families of graphs that will either behave roughly like the Bianconi--Barab\'asi or the Barab\'asi--Albert models. Therefore, our graph can be used as a tool to test the universality of these two models.

The first step is to partition our family of graphs into suitable regimes where the associated graphs share some features. Subsequently we identify, wherever possible, a representative benchmark for the class which will be either the Bianconi--Barab\'asi or the Barab\'asi--Albert.

In order to define properly the regimes, we introduce a parameter $m\in\N$ that roughly measures the total number of summed increments in~\eqref{eq:fitness_process}. More precisely, $m$ represents the length of a block of $1$'s in vector $\bm{\alpha}$. Note also that $m$ may depend on the graph size. According to the position of the block three fitness categories are identified which in turn define the following graph regimes:
\begin{defi}[R1 graph]\label{def:R1}
Let $m$ fixed. The class $R1$ is the class of all graphs $G_t$ evolving according to Algorithm~\ref{euclid} with fitness \begin{equation}\F_t(v)=\sum_{i=(t-m)\vee (v+1)}^{t}\epsilon_i(v)\end{equation}
for every node $v$.
\end{defi}
\begin{defi}[R2 graph]\label{def:R2}
Let $m$ fixed. The class $R2$ is the class of all graphs $G_t$ evolving according to Algorithm~\ref{euclid} with fitness \begin{equation}\F_t(v)=\sum_{i=v+1}^{(v+m+1)\wedge t}\epsilon_i(v)\end{equation}
for every node $v$.
\end{defi}
\begin{defi}[R3 graph]\label{def:R3}
Let $m(t)$ be a $\N$-valued function of the graph size $t$ which is increasing and unbounded. The class $R3$ is the class of all graphs $G_t$ evolving according to Algorithm~\ref{euclid} with fitness \begin{equation}\F_t(v)=\sum_{i=v+1}^{(v+m(t))\wedge t}\epsilon_i(v)\label{eq:fit_R3}\end{equation}
for every node $v$.
\end{defi}
A few remarks are now due:
\begin{rmk}[On $R1$]
The fitness process of the class $R1$ covers Cases~\ref{item:a_m}-\ref{item:b_m}. We will show that in this regime a Barab\'asi--Albert-like behavior emerges.
\end{rmk}
\begin{rmk}[On $R2$]
The fitness process of the class $R2$ covers Cases~\ref{item:d_m}-\ref{item:e_m}. We will show that in this regime a BB-like behavior emerges. Indeed, we recover a model similar to the BB in the following sense: nodes after $m$ steps stop adding increments and start behaving as if they were in a Bianconi--Barab\'asi graph with fitness law $\nu^{(m)}$.
\end{rmk}
\begin{rmk}[On $R3$]
The fitness process of the class $R3$ includes Case~\ref{item:c_m} which can be obtained by setting $m(t)=t$ in Definition~\ref{def:R3}. Clearly other functions can be used, for example $m(t)=\lfloor{\log t}\rfloor$. As we shall see, according to the speed of the chosen function different behaviors will arise especially regarding the phenomenon of condensation. One may also wonder what happens when one chooses a fitness process as
\begin{equation}
\F_t(v)=\sum_{i=(t-m(t))\vee (v+1)}^{t}\epsilon_i(v).\label{eq:fit_useless}\end{equation}
When $m(t)=t$~\eqref{eq:fit_useless} and~\eqref{eq:fit_R3} coincide. Any other other choice of $m(t)$ in~\eqref{eq:fit_useless} is uninteresting for the purposes of our study (for example, we will see that the phenomenon of condensation is trivial in this case). 
\end{rmk}
Before giving our main results, we conclude this Subsection with heuristics on the ``rich-get-richer'' phenomenon, a characterising property of preferential attachment models. 

Note that for every regime the expected fitness of vertex $v$ at time $t$ is
\begin{equation}\label{eq:ex_fitness}
    \EE[\F_t(v)]=\mu_\eps (m\wedge(t-v-1))
\end{equation}
where $m$ can also depend on $t$. It is clear from~\eqref{eq:ex_fitness} that older nodes have a higher fitness on average, thus reinforcing the idea that the old-get-richer phenomenon is likely to be preserved in our family of models.

We now turn to the main results of the paper. We will start by studying the phenomenon of condensation (defined precisely in Subsection~\ref{subsubsec:CTBP}). Then we will move to the attachment probability.
\subsection{Condensation}
The first result for condensation concerns the classical Bianconi--Barab\'asi model and shows that condensation can be enforced by a convolution operation. More specifically, given a fitness distribution in a Bianconi--Barab\'asi model with mean less than $1/2$, we will prove that the graph in Definition~\ref{def:BBm} whose fitness is the $m$-convolution condensates for $m$ large enough. On the other hand, if the mean is larger or equal to $1/2$, condensation does not appear. This result provides new insight on the heuristic behind condensation. Since this phenomenon requires more ``rarified'' high fitness population~\parencite{borgs2007first}, the convolution, which increases the tail decay rate of the distribution, will favor condensation. On the other side, if the mean is too high, mass will not escape from the upper endpoint of the distribution, countering the appearance of the condensate. Therefore, there is a trade-off between these two mechanisms, which results in a phase transition at $1/2$.

We now make the above reasoning rigorous. We introduce a family of Bianconi--Barab\'asi graphs parametrized by the convolution of the fitness law, that is, given a probability distribution $\nu$ in $[0,\,1]$, we call $BB(m)$ a Bianconi--Barab\'asi graph with fitness law $\nu^{(m)}$.
\begin{defi}[$BB(m)$ graphs]\label{def:BBm}
Let $m\in\N$. We denote $BB(m)$ a preferential graph evolving according to Algorithm~\ref{euclid} with 
\[
\F_t(v)=\sum_{i=1}^m \epsilon_i(v).
\]
\end{defi}
We stress that the fitness distribution is independent of time and is distributed as the convolution $\nu^{(m)}$. Taking $m=1$ one obtains a Bianconi--Barab\'asi model with fitness $\epsilon$. In the sequel, we will refer as ``Bianconi--Barab\'asi models'' to any graph constructed via Algorithm~\ref{euclid}, and we will use the notation $BB(m)$ when it will be important to stress the dependence on the convolution.
\begin{prop}\label{prop:force_cond}
Assume that the probability distribution $\nu$ of the fitness increment $\epsilon$ satisfies
\begin{equation}\label{eq:cond_MCT}
\int_0^1 \frac{x}{1-x}\;\de \nu(x)<\infty.
\end{equation}
and
\begin{equation}\label{eq:media_mu}
\int_0^1 x\;\de \nu(x)<\frac12.
\end{equation}
Then there exists $m^*\in\N$ such that $BB(m^*)$ condensates.
\end{prop}
\begin{rmk}[Example: Beta distribution] The law of the beta distribution $Beta(\alpha,\,\beta)$ with $\alpha<\beta<\alpha+1$ satisfies~\eqref{eq:cond_MCT}-\eqref{eq:media_mu} and does not condensate~\parencite[Appendix C.3]{borgs2007first}. Therefore by Proposition~\ref{prop:force_cond} there exists an $m^*$ such that the $m^*$-convolution condensates.
\end{rmk}
The second result instead is relative to the $R1$ regime of our model. It states that in this case condensation does not occur.
\begin{thm}\label{thm:cond}
No condensation occurs for graphs in Regime~$R1$.
\end{thm}

When trying to prove a similar result for Regime~$R2$ one faces additional difficulties. The main one is that one needs to control the empirical degree distribution, but the available techniques relying on continuous-time branching processes fail because of the interdependence among the branching rates of the particles represented by the vertices in our context. However, numerical simulations we performed in Section~\ref{sec:conj} show that when $BB(m)$ condensates, the corresponding graph in $R2$, which sums the same $m$ increments, will also condensate.  

\subsection{Evolution of the attachment probability}We now state three Propositions regarding the behavior of the attachment probability for our graphs. The main challenge lies in the fact that the attachment probability depends on the fitness, so it is a random object as well. 

In particular the role of Proposition~\ref{prop:attach_proba} is to justify why in Regime $R1$ we expect a behavior reminiscent of the Barab\'asi--Albert model. Indeed, the refresh of the fitness after $m$ steps will imply that on average we attach new nodes with a probability proportional only to the degree multiplied by a constant, the mean increment. This mirrors the behavior in the Barab\'asi--Albert graph where no fitness is present. More formally, we show
\begin{prop}~\label{prop:attach_proba} 
Let $G_t$ be a graph in Regime~$R1$. If
 $v\in V_t$ is such that $v\le t-m$, then
    \begin{equation}\label{eq:attach_R1}
    \EE\left[P(t+1\to v)\right]\asymp\frac{\EE\left[\deg_{t-m}(v)\right]}{t}+O(t^{-1})
    \end{equation}
    where the error is a.s. in $t$.
\end{prop}
Note that in~\eqref{eq:attach_R1} there is no equality sign, but we are off by a multiplicative factor as the proof will show.

Proposition~\ref{prop:attach_R2} shows that the attachment probability in Regime $R2$ depends on the fitness distribution, resulting in the naming ``Bianconi--Barab\'asi-like'' case.
\begin{prop}\label{prop:attach_R2}
Let $G_t$ be a graph in Regime~$R2$. If
 $v\in V_t$ is such that $v\le t-m$, then
    \begin{equation}\label{eq:attach_R1}
    P(t+1\to v)\asymp\frac{\deg_{t}(v)\F_t(v)}{t}\quad a.s.
    \end{equation}
\end{prop}
Finally, when $m$ depends on $t$ (the $R3$ case) we cannot refer to any benchmark model so it is natural to investigate the attachment probability in this case too. In particular we observe a behavior more similar to the Barab\'asi--Albert model. This is due to the fact that, essentially as a consequence of the law of large numbers, the fitness may be replaced by a constant, its mean, thus cancelling out in the numerator and denominator of the attachment probability.
\begin{prop}\label{item:nr_tre}
Let $G_t$ be a graph in Regime~$R3$. Then
    \begin{equation}\label{eq:attach_R3}
    P(t+1\to v)\asymp \frac{\deg_t(v)}{t}\quad a.s.\text{ and }L^1.
    \end{equation}
    \end{prop}
\begin{rmk}
We notice in passing that when the fitness is distributed as in~\eqref{eq:fit_useless} the result of Proposition~\ref{item:nr_tre} carries over as well.
\end{rmk}

\section{Proofs of the results}\label{sec:proofs}
\subsection{Preliminaries}
The two main tools we use to study preferential attachment graphs are the majorization order and some results in the theory of branching processes. Although by no means complete, we wish to recall here the basics we are going to employ in our work.
\subsubsection{Majorization} Majorization is a tool which was first introduced in~\cite{hardy1929some}. We refer the interested reader to the monography~\cite{marshall1979inequalities} for a complete overview. 
\begin{defi}[Majorization order]
For two vectors $\mathbf u,\,\bo{v}\in\R^d_+$ we will write $\mathbf u\prec \bo v$ (``$v$ majorizes $u$''), if and only if the following is satisfied:
\[
 \left\{\begin{array}{cc}
   \sum_{i=1}^k u_{[i]}\le \sum_{i=1}^k v_{[i]}, & k=1,\,\ldots,\,d-1  \\
   \sum_{i=1}^d u_{[i]}=\sum_{i=1}^d v_{[i]}.  & 
\end{array} \right.
\]
\end{defi}
We will call
\[
\mathscr D_d:=\{(u_1,\,\ldots,\,u_d)\in \R^d:\,u_1\ge u_2\ge \ldots\ge u_d\}.
\]
Majorization becomes a useful tool in random graphs because it provides a way to control functions whose domain is a simplex, and since the degree sequence satisfies $\sum_{v=1}^t\deg_t(v)=2(t-1)$ for trees we can apply majorization to find maxima and minima of appropriate quantities. In particular, we will look at Schur-convex functions, which are are isotonic with respect to the majorization order (see~\cite[Definition A.1,~Section 3]{marshall1979inequalities}). One such function is the partition function of the attachment probability at time $t$.
\subsubsection{Condensation and continuous-time branching processes}\label{subsubsec:CTBP} Condensation can be defined rigorously in several ways. The first definition we use requires the introduction of the upper end-point of the fitness distribution $\nu$: 
$$ h=h(\F):=\sup\{x:\,\nu(-\infty,\,x)<1\}.$$
To restrict ourselves to interesting cases, we assume there is no atom at the upper end-point, that is, $\nu\{h(\F)\})=0.$
The standard approach to study condensation is the embedding of preferential attachment graphs into continuous-time branching processes.~This technique goes back to~\cite{athreya1968embedding,bhamidi2007universal,janson2004functional} and we adapt here the presentation given in~\cite{dereich2017nonextensive} to the setting of the Bianconi--Barab\'asi model with fitness law $\nu$ supported on $[0,\,1]$.

In~\cite[Theorem~2.1]{dereich2017nonextensive} it is shown that a Bianconi--Barab\'asi model exhibits condensation if 
\begin{equation}\label{eq:cond_cecile_dereich}
\int_0^1\frac{\F}{1-\F} \de  \F<1.
\end{equation}
In this case the {\em weighted empirical} fitness distribution
\[
\Xi_t:=\frac{1}{2t}\sum_{i=1}^t \deg_t(i)\delta_{\F_t(i)}
\]
converges as $t$ goes to infinity to the sum of an absolutely continuous\footnote{With respect to the fitness law.} part,
called the bulk, and a Dirac mass in the essential supremum of the support of the fitness distribution, called the condensate. 

By viewing the CTBP as a reinforced P\'olya's urn, it is also possible to study condensation by establishing the strict positivity in the limit of the cumulative degree for vertices with high fitness. This is in fact the first approach to the mathematical study of condensation, pioneered by~\cite{borgs2007first}, see also~\cite{freeman2018extensive}. Let $V_{\tilde h}:=\{v\in V_{t}:\,\F_t(v)\ge \tilde h\}$ for $\tilde h\le h$. Condensation is based on the behavior of the functional 
\begin{equation}
M_{\tilde h}:=\sum_{v\in V_{t}}\deg_{t}(v)\mathbbm{1}_{\{\F_t(v)\ge \tilde h\}}
\end{equation}
as $t\to \infty$ (firstly) and $\tilde h\to h$ (secondly). Models that do condensate are those for which
\begin{equation}\label{eq:cond_cond}
 \lim_{\tilde h \to h}\liminf_{t\to\infty}\EE\left[\frac{M_{\tilde h}}{t}\right]>0.
\end{equation}
We remark that these conditions to define condensation have been investigated only in the cases of static fitnesses. It is for example possible to see without difficulty that~\eqref{eq:cond_cond} is satisfied for $\F_t$ such that $\lim_{t\to\infty}\F_t=+\infty$ a.s. and for which an $m$-Markov property holds, as one first takes the limit in $t$ and then in $\tilde h$. This is the reason why we are not interested in studying models with fitness process given in~\eqref{eq:fit_useless}.

\subsection{Auxiliary lemmas}
\subsubsection{Condensation}
The next Lemma shows that for the $BB(m)$ model defined in Definition~\ref{def:BBm} condensation is monotone under the convolution operation. Namely, once observed, the phenomenon of condensation is not disrupted by adding more increments in the fitness.
\begin{lemma}Assume $\nu$ has compact support and $m_1<m_2$. If $BB(m_1)$ condensates then $BB(m_2)$ condensates. 
\end{lemma}
\begin{proof}
Without loss of generality we can assume $m_1=1,\,m_2=2$. The proof will proceed similarly for $1<m_1<m_2$ by a repeated application of the arguments below. We assume also that the law $\nu$ of the fitness of $BB(1)$ is normalised so that $\text{supp}(\nu)=[0,\,1]$. In order to prove the result we will verify the condition of condensation~\eqref{eq:cond_cecile_dereich}. We have, by assumption,
\begin{equation}\label{eq:given}
\int_0^1 \frac{x}{1-x}\de\nu(x)<1
\end{equation}
and we have to show
\begin{equation}\label{eq:to_show_fit}
\int_0^2 \frac{x}{2-x}\de\nu^{(2)}(x)<1.
\end{equation}
Mind that under $\nu^{(2)}$ one has $X=X_1+X_2 $ with each $X_i\sim \nu$ and independent from each other. We introduce the notation $\bo x_m:=(x_1,\,\ldots,\,x_m)$. We rewrite~\eqref{eq:to_show_fit} as
\[
\int_{0}^2 \frac{\frac{x}{2}}{1-\frac{x}{2}}\de\nu^{(2)}(x)<1\iff \iint_{[0,\,1]^2} \frac{\langle \bo a_2,\, \bo x_2\rangle}{1-\langle \bo a_2,\, \bo x_2\rangle}\de\nu(x_1)\de \nu(x_2)<1
\]
where $\bo a_m:=(\underbrace{\nicefrac{1}{m},\,\ldots,\,\nicefrac{1}{m}}_{m},\,0,\,\ldots,\,0)\in \R^{m+\ell}$, $\ell\in\N$. In the present case we choose $m=2,\,\ell=0$.
Consider the function
\begin{align}\phi:\,\mathscr D_2&\to \R\nonumber\\
\bo a_2&\mapsto \iint_{[0,\,1]^2} \frac{\langle \bo a_2,\, \bo x_2\rangle}{1-\langle \bo a_2,\, \bo x_2\rangle}\de\nu(x_1)\de \nu(x_2). \label{eq:def_phi} \end{align}
This function is Schur-convex in $\mathscr D_2$ as one can see by applying~\cite[Theorem~A.3,~Section~3]{marshall1979inequalities}:
\begin{equation}\label{eq:Schur_conv}
\frac{\partial \phi}{\partial \bo (\bo a_2)_i}(\bo a_2)=\1_{\bo (\bo a_2)_i\neq 0}\iint_{[0,\,1]^2}\frac{x_i }{(1-\langle \bo a_2,\, \bo x_2\rangle)^2}\de\nu(x_1)\de \nu(x_2)\ge 0,\quad i=1,\,2.
\end{equation}
The derivative in each $  (\bo a_2)_i$ can be taken inside the integral since the integrand is $C^1$ in the domain $ \mathscr D_2\times [0,\,1)^2$. Note also that $\frac{\partial \phi}{\partial   (\bo a_2)_1}(\bo a_2)=\frac{\partial \phi}{\partial   (\bo a_2)_2}(\bo a_2).$ Therefore by definition of Schur-convexity and the fact that $\bo a_m\succeq \bo a_{m-1}$ it follows that
\[
 \phi(\bo a_2)\le \phi(\bo a_1)=\int_0^1 \frac{x_1}{1-x_1}\de \nu(x_1)\stackrel{\eqref{eq:given}}{<}1
\]
which implies~\eqref{eq:to_show_fit}.
\end{proof}
As a part of the proof (and we will name it Corollary) we have obtained the following:
\begin{cor}\label{cor:phi}
For $i,\,j\in\N$, the function $\phi$ of~\eqref{eq:def_phi} satisfies $\phi(\bo a_i)\le \phi(\bo a_j)$ if $i\ge j$.
\end{cor}
\begin{proof}[Proof of~Proposition~\ref{prop:force_cond}]
Without loss of generality we assume that $\nu$ is such that $BB(1)$ does not exhibit condensation (otherwise $m^*=1$). To show that the model condensates for some $m^*\in\{2,\,3,\,\ldots\}$, we observe that for a random vector $\bo X_m\in[0,\,1]^m$ with $\bo X_m\sim \prod_{i=1}^m\de \nu$ one has \begin{equation}\label{eq:average_kernel}\langle  \bo a_m,\,\bo X_m\rangle=\overline X_m=\sum_{i=1}^m \frac{X_i}{m}.\end{equation}
Now notice that~\eqref{eq:cond_cecile_dereich} can be rewritten using~\eqref{eq:average_kernel} as
\begin{equation}\label{eq:rearr_x}
\EE\left[\frac{\overline X_m}{1-\overline X_m}\right]=\EE\left[\frac{\overline{(\bo X_\downarrow)}_m}{1-\overline{(\bo X_\downarrow)}_m}\right]
\end{equation}
since under $\prod_{i=1}^m\de \nu$ one has
\begin{equation}\label{eq:subs_LLN}
\overline X_m\stackrel{d}{=}\overline{(\bo X_\downarrow)}_m=\frac1m\sum_{i=1}^m X_{[i]}.
\end{equation}
From $\bo X_\downarrow$ fix a realisation $\bo x_\downarrow$. Consider the function
\begin{align}\Phi:\,\mathscr D_m&\to \R\nonumber\\
\bo a&\mapsto  \frac{\langle \bo a,\,\bo x_\downarrow\rangle}{1-\langle \bo a,\,\bo x_\downarrow\rangle}. \label{eq:def_Phi} \end{align} 
As shown for $\phi$ of~\eqref{eq:def_phi} in~\eqref{eq:Schur_conv}, one can prove that $\Phi$ is Schur-convex, so that, for fixed $x_\downarrow$, $\Phi(\bo a_m)$ is decreasing in $m$. This enables us to say that
\begin{equation}\label{eq:last_eq}
\lim_{m\to\infty}\EE\left[\frac{\overline X_m}{1-\overline X_m}\right]\stackrel{\eqref{eq:rearr_x}}{=}\lim_{m\to\infty}\EE\left[\frac{\overline{(\bo X_\downarrow)}_m}{1-\overline{(\bo X_\downarrow)}_m}\right]=\EE\left[\lim_{m\to\infty}\frac{\overline{(\bo X_\downarrow)}_m}{1-\overline{(\bo X_\downarrow)}_m}\right]
\end{equation}
using the monotone convergence theorem in the last step (applicable by~\eqref{eq:cond_MCT} and the monotonicity of $\Phi$). We can now show the right-hand side of~\eqref{eq:last_eq} equals $\mu_\eps/(1-\mu_\eps)$ using~\eqref{eq:subs_LLN} and the strong law of large numbers. This yields that $\left\{\EE\left[\frac{\overline X_m}{1-\overline X_m}\right]:\,m\in\N\right\}$ is a bounded decreasing sequence converging to $\mu_\eps/(1-\mu_\eps)$. This implies the result.
\end{proof}

\subsubsection{Attachment probability}
A classical result we need to quote is the following. Its proof can be found in~\cite[Theorem 368, Section 10.2]{hardy1952inequalities}.
\begin{lemma}[Rearrangement inequality]\label{lem:rearr}
For every $n\in\N$, every sequence of real numbers $x_1\le x_2\le\ldots\le x_n$, $y_1\le y_2\le\ldots\le y_n$ and every permutation $\sigma\in\mathfrak S_n$ it holds that
\[
\sum_{i=1}^n y_i x_{n-i+1}\le \sum_{i=1}^n y_i  x_{\sigma(i)}\le \sum_{i=1}^n y_i x_i.
\]
\end{lemma}
In order to treat the attachment probability, we need to have control on the partition functions. We will do so using majorization in Regimes~$R1$-$R3$.
\begin{lemma}\label{lem:ristorante}
Let $Z_t:=\sum_{v\in V_{t-1}}\F_{t}(v)\deg_{t-1}(v)$ be the partition function of~\eqref{eq:attach_rule} of models in Regimes $R1$-$R3$. Then the following holds:
\[
(Z_t)^{-1}\asymp (t m)^{-1}\quad a.s.
\]
where the constants in the asymptotic upper and lower bounds are deterministic.
\end{lemma}
\begin{proof}
The proof is based on the following two steps:
\begin{enumerate}[leftmargin=*]
    \item first we find two matching a.s. upper and lower bounds for $Z_t$ that involve roughly the same sum of independent and identically distributed random variables.
    \item Secondly we show by the strong law of large numbers that the sum behaves asymptotically like $mt$.
\end{enumerate}
Let us begin with the two bounds. Using the fact that the degree is always at least one, we can bound $Z_t$ from below by
\begin{align}\label{eq:Z_bound_below}
        Z_t\ge \sum_{v\in V_{t-1}} \F_t(v).
    \end{align}
We then look for a similar bound from above. Let us rename the vertices $v\in V_{t-1}$ in such a way that $\deg_t(v)=\deg_{t,\,\downarrow}(v).$ Namely we rearrange the degree sequence in decreasing order.
Therefore for some permutation $\sigma\in\mathfrak S_t$
\begin{equation}\label{eq:Maj_Z}
Z_t=\sum_{v\in V_{t-1}}\deg_{t,\,\downarrow}(v) \F_t(\sigma(v))\le \sum_{v\in V_{t-1}}\deg_{t,\,\downarrow}(v) \F_{t,\,\downarrow}(v)
\end{equation}
by the Rearrangement Inequality of Lemma~\ref{lem:rearr}. Now $(\deg_{t,\,\downarrow}(v))\in \mathscr D^*:=\{\bo x\in\mathscr D_t:\,\sum_{i}x_i=2(t-1)\}.$ The function
\[
\mathscr D^*\ni \bo x\mapsto \langle \bo x,\, \bo{\F}_{t,\,\downarrow}\rangle
\]
is a.s. Schur-convex in $\mathscr D^*$ by~\cite[Theorem~A.3,~Section~3]{marshall1979inequalities}. Therefore it attains its maximum at the maximal element for the majorization order in the simplex $\mathscr D^*$~\parencite[Prop. H.2.a,~Section~3]{marshall1979inequalities}. It is straightforward to identify this element as
    \begin{align}\label{claim:appendix}
   \bo x_{max}&=(t-1,\,\underbrace{1,\,\ldots,\,1}_{t-1}).  
    \end{align}
    Given~\eqref{claim:appendix}, we invoke~\eqref{eq:Maj_Z} to argue that with probability one
    \begin{equation}\label{eq:maj_up_Z}
    Z_t\le (t-1)\F_{t,\,\downarrow}(1)+\sum_{v=2}^t\F_{t,\,\downarrow}(v)\le (t-1)m+ \sum_{v=1}^t\F_{t}(v).
     \end{equation}
     We have thus obtained~\eqref{eq:Z_bound_below} and \eqref{eq:maj_up_Z} which have the same order of magnitude as $t$ grows. We will then study only the asymptotics for~\eqref{eq:Z_bound_below}, the other bound being very similar.

We begin by observing that, for every $v\in V_{t-1},$
    \begin{align}\label{eq:sum_fitness}
        Z_t\ge \sum_v \F_t(v)\ge \sum_{k=1}^L\epsilon_k
    \end{align}
where we have relabeled the increments\footnote{In~\eqref{eq:sum_fitness} we are summing all increments in the tree. Therefore we drop the dependence of $\eps$ on a vertex $v$ since the increments are i.i.d.} and numbered them until
\begin{equation}\label{eq:L}L:=(t-m)m+(m-1)m/2=tm-m^2/2-m/2.\end{equation}
Equation~\eqref{eq:sum_fitness} holds for any regime because the total number of increments is the same. Equation~\eqref{eq:L} is going to infinity for $m\le t$. Then by the strong law of large numbers, for every $\eta>0$ we can find an a.s. $L_0=L_0(\eta)$ such that for all $L\ge L_0$ $$\sum_{k=1}^L\epsilon_k\ge \mu_\epsilon L-\eta.$$
Choose then $t_0=t_0(\eta)$ in a set of probability one such that $L\ge L_0$ for $t\ge t_0$ (this is possible since $L$ is an explicit function of $t$). Therefore we obtain an almost-sure bound of the form
\begin{equation}\label{eq:boundFit}
Z_t\ge \sum_v \F_t(v)\ge \mu_\epsilon(tm-m^2/2-m/2)-\eta.
\end{equation}
This concludes the proof.
\end{proof}

\subsection{Proofs of the main results}
\begin{proof}[Proof of Theorem~\ref{thm:cond}]~
For a preferential attachment model with fitness, we have that
\begin{align}
M_{\tilde h}&=\sum_{v\in V_{t-1}}\deg_{t-1}(v)\1_{\{\F_t(v)\ge \tilde h\}}=\sum_{v}\deg_{t-1-m}(v)\1_{\{\F_t(v)\ge \tilde h\}}\nonumber\\
&+\sum_{v}D_{v,\,m}\1_{\{\F_t(v)\ge \tilde h\}}\label{eq:dec_M}
\end{align}
where $D_{v,\,m}:=\{t'\ge t-m:\,t'\to v\}$ is a random variable bounded almost surely by $m$. Thus recalling the definition $V_{\tilde h}:=\{v\in V_{t}:\,\F_t(v)\ge \tilde h\}$, we have that $\EE[M_{\tilde h}]$ is bounded above by
\begin{align}
P(\F_t(1)\ge \tilde h)&\EE\left[\sum_{v}\deg_{t-m-1}(v)\right]+m \EE[V_{\tilde h}]\nonumber\\
&=P(\F_t(1)\ge \tilde h)2(t-m-2)+m \EE[V_{\tilde h}].\label{eq:fitn_RHS}
\end{align}
Here we have used the fact that $\F_t(1)$ is independent of the sigma algebra $\sigma(G_{t-m-1})$ (the fitness has the $m$-Markov property) and that the sum of the degrees up to $t-m-1$ is deterministic. Furthermore, we notice that, due to the independence of the fitnesses over vertices, one has
\begin{equation}\label{eq:card_V}
    \frac{\EE[V_{\tilde h}]}{ t-1}\sim P\left(\sum_{i=1}^m \epsilon_i\ge \tilde h\right)\to 0
    \end{equation}
    as $t\to \infty$ and $\tilde h \to h$. We justify~\eqref{eq:card_V} since for every node $v\le t-m$ the fitness is a sum of $m$ i.i.d. increments. These two observations combined prove that, for $m$ constant,~\eqref{eq:fitn_RHS} converges to $0$.\qedhere
\end{proof}

\begin{proof}[Proof of Proposition~\ref{prop:attach_proba}]~
We recall the bound
    \begin{equation}\label{eq:bound_on_Zt}
    \mu_\eps m\left(t-\frac{m}{2}\right)(1+o(1))\le Z_t\le \mu_\eps m\left(t-\frac{m}{2}\right)(1+o(1))+tm\quad a.s.
    \end{equation}
    for $t$ large enough from Lemma~\ref{lem:ristorante}. Thus using the left-hand side of~\eqref{eq:bound_on_Zt} one can rewrite the expected attachment probability as
    \begin{align}
 \EE\left[\frac{\deg_{t-1}(v)\F_t(v)}{Z_t}\right]&\ge \EE\left[\frac{\deg_{t-1}(v)\F_t(v)}{\mu_\eps m \left(t-\frac{m}{2}\right)(1+o(1)) }\right]=\EE\left[\frac{\deg_{t-m}(v)\F_t(v)}{\mu_\eps m \left(t-\frac{m}{2}\right)(1+o(1)) }\right]\nonumber\\
 &+\EE\left[\frac{D_{v,\,m}\F_t(v)}{\mu_\eps m \left(t-\frac{m}{2}\right)(1+o(1)) }\right]\label{eq:dec_attach_prob}
    \end{align}
    where $D_{v,\,m}$ is as in~\eqref{eq:dec_M}. The a.s. bound on $D_{v,\,m}$ and the fact that $\EE[\F_t(v)]=\mu_\eps m$ yield that the second summand in~\eqref{eq:dec_attach_prob} is $O(1/t)$ with probability one. As for the first summand, note that $\deg_{t-m}(v)$ and $\F_t(v)$ are independent. Therefore we obtain that
     \begin{align*}
 \EE\left[\frac{\deg_{t-m}(v)\F_t(v)}{Z_t}\right]\ge \EE\left[\frac{\deg_{t-m}(v)}{\left(t-\frac{m}{2}\right)(1+o(1))}\right].
 \end{align*} 
 The other bound can be obtained in the same way from the right-hand side of~\eqref{eq:bound_on_Zt}.
 \end{proof}
 \begin{proof}[Proof of Proposition~\ref{prop:attach_R2}]
 The result follows by applying Lemma~\ref{lem:ristorante} to the partition function of the attachment probability.
 \end{proof}
 \begin{proof}[Proof of Proposition~\ref{item:nr_tre}]
 Using again the right-hand side of~\eqref{eq:bound_on_Zt} we get
 \begin{equation}\label{eq:boundR3_attach}
 \frac{\deg_{t-1}(v)\F_t(v)}{Z_t}\ge \frac{\deg_{t-1}(v)}{(\mu_\eps +1)tm(1+\frac{\mu_\eps}{2(\mu_\eps+1)}\frac{m}{t}+o(1))}\frac{\F_t(v)}{\mu_\eps m}\mu_\eps m.
 \end{equation}
 Observe now that \[
 \frac{\deg_{t-1}(v)}{t}\le 1
 \] 
 and that ${\F_t(v)}/({\mu_\eps m})$ converges to one by the strong law of large numbers with probability one. Hence by dominated convergence theorem one can argue that
\[
\lim_{t\to\infty}\frac{\mu_\eps }{\mu_\eps +1}\EE\left[\left|\frac{\deg_{t-1}(v)}{t(1+\frac{\mu_\eps}{2(\mu_\eps+1)}\frac{m}{t}+o(1))}\left(\frac{\F_t(v)}{\mu_\eps m}-1\right)\right|\right]=0.
\]
This shows that
\[
\frac{\deg_{t-1}(v)\F_t(v)}{Z_t}\ge X^{(\ell)}_t
\]
where $X^{(\ell)}_t$ is a random variable that is asymptotic in $L^1$ and a.s. to $$
\frac{\mu_\eps}{\mu_\eps+1}\frac{ \deg_t(v)}{t(1+\frac{\mu_\eps}{2(\mu_\eps+1)}\frac{m}{t})}.
$$
The a.s. statement is a consequence of the law of large numbers for the term ${\F_t(v)}/({\mu_\eps m})$ going to one.
A similar upper bound, this time using the lower bound of the partition function in~\eqref{eq:bound_on_Zt}, yields that
\[
\frac{\deg_{t-1}(v)\F_t(v)}{Z_t}\le X^{(r)}_t
\]
where $X^{(r)}_t$ is a random variable that is asymptotic in $L^1$ and a.s. to
$$
\frac{ \deg_{t-1}(v)}{\left(t-\frac{m}{2}\right)}.
$$
\end{proof}
\section{Conjectures}\label{sec:conj}
The condensation phenomenon and the attachment probability hint at the fact that the Barab\'asi--Albert and the $BB(m)$ models represent benchmarks. However, these two quantities are not sufficient to establish a full universality result. Therefore we believe that investigating other aspects of interest can strengthen our claim. We will devote this Section to the numerical study of some additional observables of our graph and the relation with the benchmark models.

We focus on the degree distribution and the condensation phenomenon. As for the former, since the results on the attachment probability are local, in the sense that they hold for fixed vertices, looking at the degree distribution gives broader information on the network. As for the latter, we want to verify whether the threshold for the appearance of condensation derived in Proposition~\ref{prop:force_cond} for the $BB(m)$ is mirrored in our model in Regime $R2$. 

Finally, since to the best of the authors' knowledge there is no reference in the literature to preferential attachment models with fitness as in Regime $R3$, we want to shed light on the behavior of condensation in this case.

\subsection{Degree distribution}
We now present the numerical results for the degree distribution in Regimes~$R1$ and $R2$. 

We will show that the total variation distance between the degree distribution of our model and the benchmarks vanishes asymptotically in the graph size. We chose the total variation distance, other than for its numerical tractability, also because it implies the convergence of the laws.
\begin{conj}[Convergence of the degree distribution in $R1$]\label{conj:d_TVAB}
Let $m\in \N$, and let $\mathbf{deg}_t$ be the empirical degree distribution of a graph $G_t$ in Regime~$R1$. Let $\mathbf{deg}_t^{BA}$ be the empirical degree distribution of a Barab\'asi--Albert graph. Then
\begin{equation}\label{eq:dTVAB}
\lim_{t\to\infty}d_{TV}(\mathbf{{deg}}_t,\,\mathbf{deg}_t^{BA})=0.
\end{equation}
The limit is taken a.s. in the fitness realisation for $G_t$.
\end{conj}
In Figure~\ref{fig:dTVMean} we plot the mean of the total variation distance~\eqref{eq:dTVAB} averaged over 100 Monte Carlo simulations with $m=1$ and $\eps\sim U(0,\,1)$ for different graph sizes. Since our result is quenched in the fitness, we are keeping the same realization of the fitnesses and averaging the total variation distance over the Monte Carlo trials.

\begin{figure}[ht!]
  \begin{minipage}[t]{0.45\textwidth}
        \includegraphics[width=\textwidth]{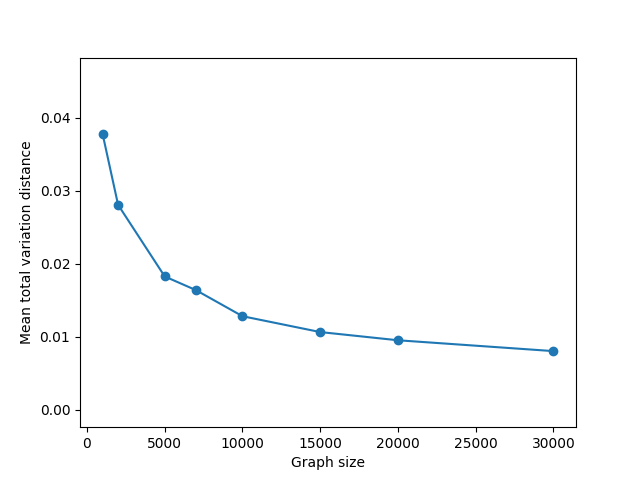} 
        \end{minipage}
        \begin{minipage}[t]{0.45\textwidth}
        \includegraphics[width=\textwidth]{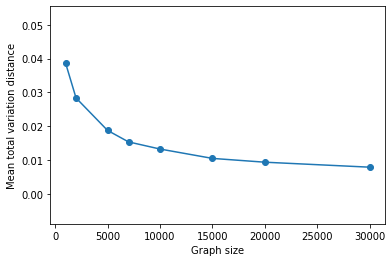} 
        \end{minipage}
        \caption{Left: mean of the total variation distance~\eqref{eq:dTVAB}, $m=1$, $\epsilon\sim U(0,\,1)$. Right: mean of the total variation distance~\eqref{eq:dTVAB}, $m=5$, $\epsilon\sim U(0,\,1)$. Note that it decreases towards zero.}
        \label{fig:dTVMean}
\end{figure}

Due to the convergence of~\eqref{eq:dTVAB} we are also conjecturing that the asymptotic survival function of the degree distribution is close to a power law with exponent $\tau=2$, as in the standard Barab\'asi--Albert model~\parencite[Section~8.4]{remco2016random}. We then compare the tail exponent of the survival function of the degree distribution between our model in $R1$ and the Barab\'asi--Albert model in Figure~\ref{fig:tail_expABNostro3}.
\begin{figure}[ht!]
    \centering
    \includegraphics[width=.45\textwidth]{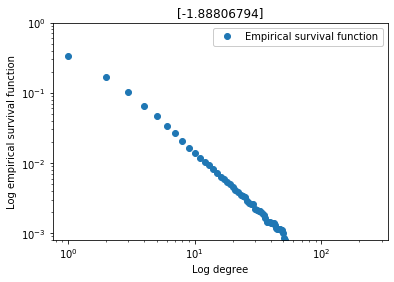}
    \includegraphics[width=.45\textwidth]{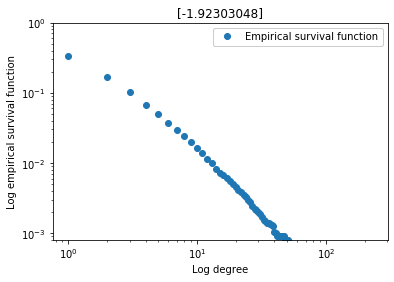}
    \caption{Loglog plot of the empirical survival function for our model (left) in Regime~$R1$, $m=1$, $\epsilon\sim Beta(1,\,3)$ and a Barab\'asi--Albert model (right). The plot looks like a straight line which hints towards a power law behavior~\parencite{cirillo2013your}. On top $\tau$ is computed.}
    \label{fig:tail_expABNostro3}
\end{figure}

\begin{conj}[Convergence of the degree distribution in $R2$]\label{conj:d_TVBB}
Let $m\in\N$, and let $\mathbf{deg}_t$ be the empirical degree distribution of a graph $G_t$ in Regime~R2. Let $\mathbf{deg}_t^{BB(m)}$ be the empirical degree distribution of a $BB(m)$ model with the same parameter $m\in\N$. Then
\begin{equation}\label{eq:dTVBB}
\lim_{t\to\infty}d_{TV}(\mathbf{deg}_t,\,\mathbf{deg}_t^{BB(m)})=0.
\end{equation}
The limit is taken a.s. in the fitness realisation of $G_t$.
\end{conj}
In Figure~\ref{fig:dTVBBN1} we plot the mean total variation distance~\eqref{eq:dTVBB} over 100 Monte Carlo simulations with $m=5$ resp. $m=10$ and $\epsilon\sim Beta(1,\,3)$ for various graph sizes. As in the AB case, the fitness realisation is kept fixed over the various Monte Carlo trials.
\begin{figure}[ht!]
    \begin{minipage}{0.45\textwidth}
        \includegraphics[width=1\textwidth]{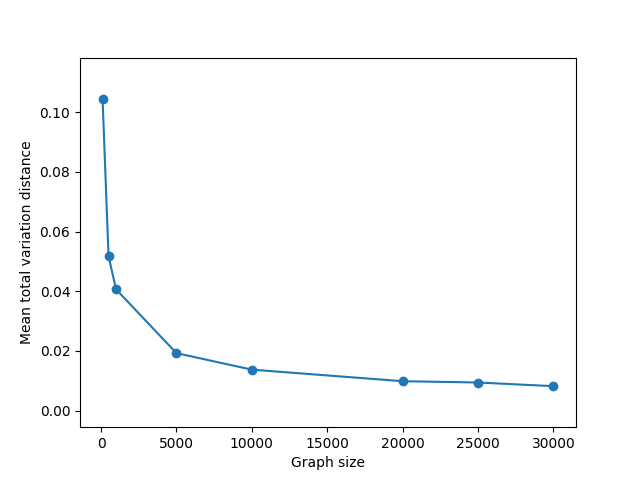} 
        \end{minipage}
        \begin{minipage}{0.45\textwidth}
        \includegraphics[width=1\textwidth]{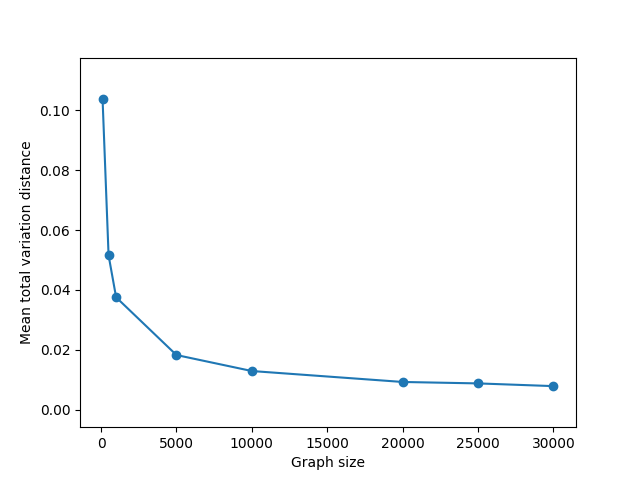} 
        \end{minipage}
        \caption{Left: mean of the total variation distance \eqref{eq:dTVBB} between our model in Regime~$R2$ and $BB(m)$, $m=5$, $\epsilon\sim Beta(1,\,3)$. Right: mean of the total variation distance \eqref{eq:dTVBB} between our model in Regime~$R2$ and $BB(m)$, $m=10$, $\epsilon\sim Beta(1,\,3)$. In both cases it goes to zero.}
        \label{fig:dTVBBN1}
\end{figure}

As a comparison, observe in Figure~\ref{fig:wrongdTV} the behavior of the mean total variation distance between our model in Regime~$R2$ with $m\neq 1$ and a $BB(1)$ averaged over $100$ Monte Carlo simulations.
\begin{figure}[ht!]
    \centering
    \begin{minipage}{.45\textwidth}
    \includegraphics[width=1\textwidth]{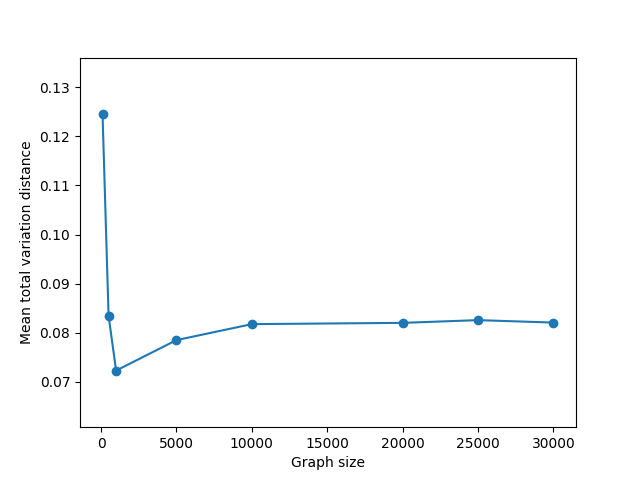}
    \end{minipage}
    \begin{minipage}{.45\textwidth}
    \includegraphics[width=1\textwidth]{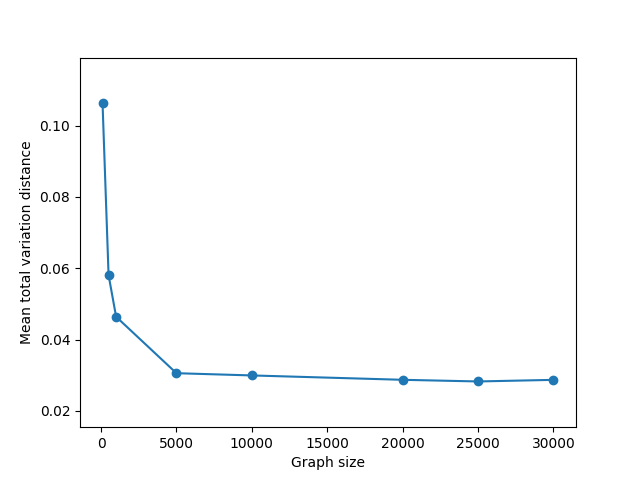}
    \end{minipage}
    \caption{Left: Mean total variation distance~\eqref{eq:dTVBB} between our model with $m=5$ and a $BB(1)$, $\epsilon\sim Beta(1,\,3).$ Right: mean total variation distance~\eqref{eq:dTVBB} between our model with $m=2$ and a $BB(1)$, $\epsilon\sim U(0,\,1).$ In both cases the mean total variation does not approach zero.}
    \label{fig:wrongdTV}
\end{figure}
Again this hints at the fact that the way in which fitnesses increase is substantially uninfluential on the growth of the network, provided we sum finitely many increments. This also shows that the $BB(m)$ model is robust under dynamical perturbations.
\subsection{Condensation}\label{subsec:cond}
We now present the numerical results on condensation in Regimes~$R1$-$R3$. To do so, we will plot the cumulative degree of the nodes grouped by fitness. In this setting, based on the argument outlined in the Introduction, we expect to see condensation when the landscape of the above has more and higher spikes concentrated towards the upper end point of the fitness law.

To begin with, recall that by Theorem~\ref{thm:cond} no condensate appears in Regime~$R1$. Indeed, when picturing condensation using the cumulative degree grouped by fitness (see Figure~\ref{pics:tramaxAB}) one can notice that the position of the spikes varies on the whole support of the distribution. This is due to the the $m$-Markov property. 

\begin{figure}[ht!]
\centering
\includegraphics[width=.3\textwidth]{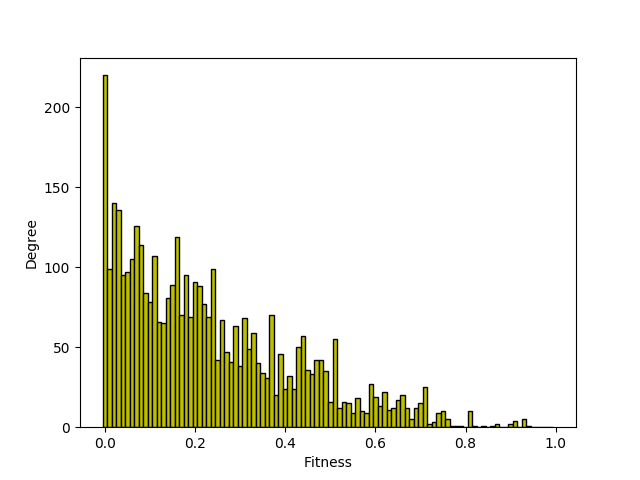}\quad
\includegraphics[width=.3\textwidth]{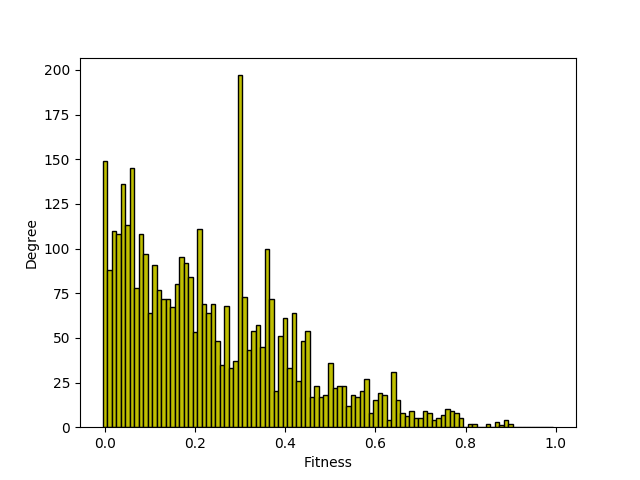}\quad
\includegraphics[width=.3\textwidth]{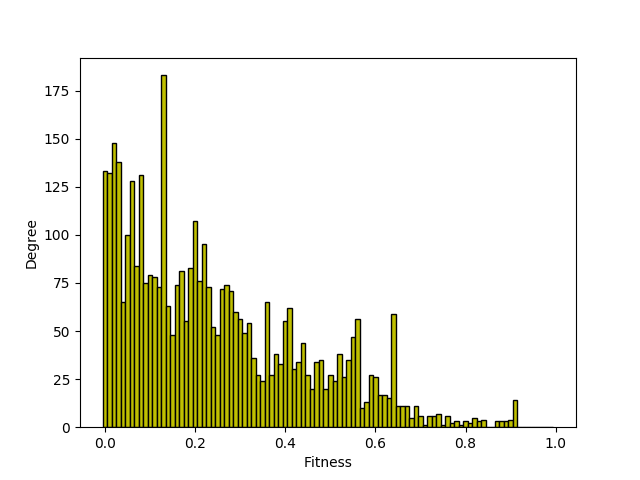}



\caption{Cumulative degree of the nodes grouped by fitness for our model in Regime~$R1$ with $m=1$, $\epsilon\sim Beta(1,\,3)$. Observe that the location of the highest spike of the degree does not escape towards the supremum of the fitness but is randomly shuffled.
}
\label{pics:tramaxAB}
\end{figure}

We now turn our attention to Regimes $R2$ and $R3$. 
\subsection{Condensation for $R2$}  
Given Conjecture~\ref{conj:d_TVBB} on the asymptotic degree distribution, we formulate a conjecture on the condensation for $R2$ models. 
\begin{conj}
Let $m\in\N$ and let $G_t$ be a preferential attachment graph in Regime~$R2$. Then $G_t$ condensates if $BB(m)$ condensates.
\end{conj}
We will support the above conjecture with a few simulations. We recall~\parencite[Appendix C.3]{borgs2007first} that in a $BB(1)$ model the fitness distribution $Beta(\alpha,\,\beta)$ condensates if and only if $\beta>\alpha+1$. In Figure~\ref{pics:tramax_m1_Beta11} one can observe the absence of a condensate for $U(0,\,1)$-distributed increments and $m=1$. In Figure~\ref{pics:tramax_m5_Beta119} a condensate appears for the increment distribution $Beta(1,\,1.9)$ when $m=2$. Note finally in Figure~\ref{fig:phase_trans_nostro} that there is no condensation for
$Beta(3,\,1)$-distributed increments for $m=5$. Since in this case $\mu_\eps=3/4>1/2$ the threshold in Proposition~\ref{prop:force_cond} seems to be binding in $R2$ as well.
\newpage
\begin{figure}[ht!]
\centering
\includegraphics[width=.25\textwidth]{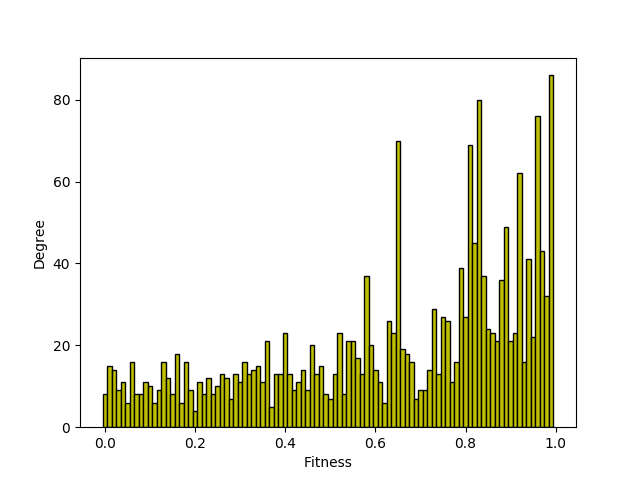}\quad
\includegraphics[width=.25\textwidth]{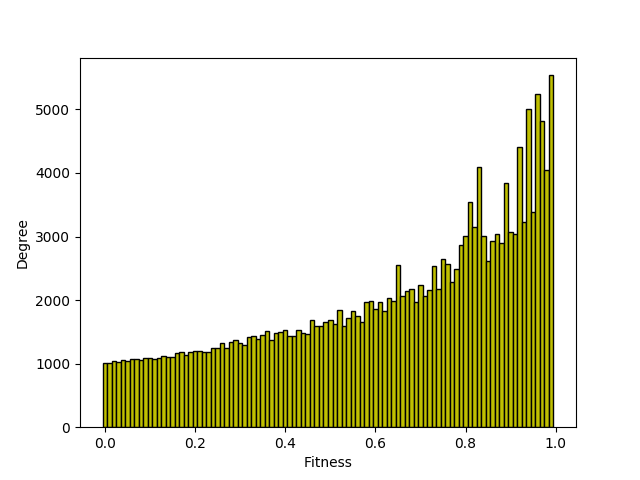}\quad
\includegraphics[width=.25\textwidth]{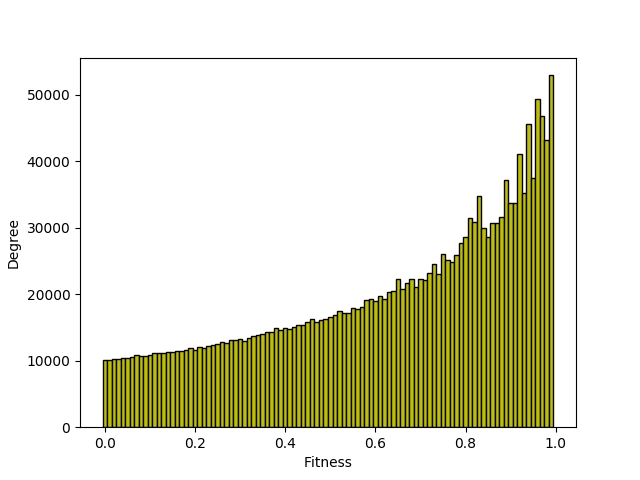}

\caption{Cumulative degree in fitness intervals for Regime~$R2$ with $m=1$, $\epsilon\sim U(0,\,1)$. Note that the degree distribution per fitness becomes smoother at the right endpoint of the fitness distribution. This model is known not to exhibit condensation. 
}
\label{pics:tramax_m1_Beta11}
\end{figure}
\begin{figure}[ht!]
\centering
\includegraphics[width=.25\textwidth]{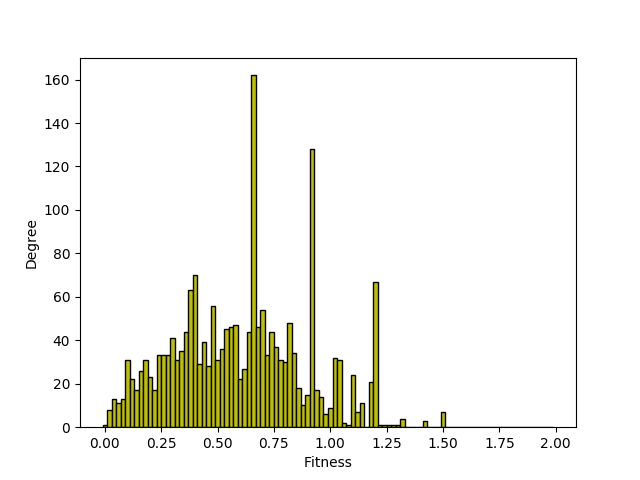}\quad
\includegraphics[width=.25\textwidth]{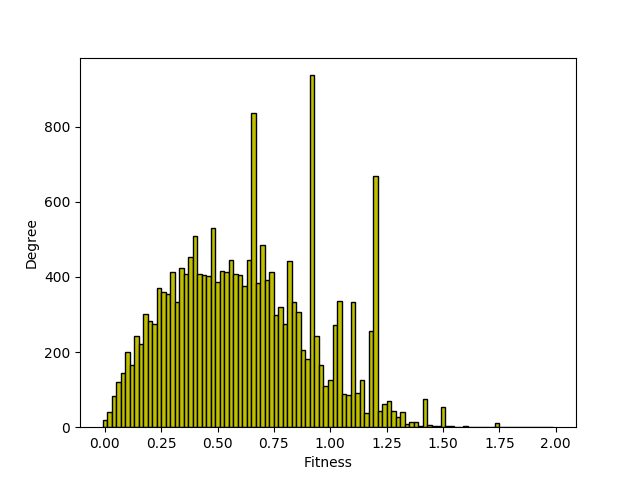}\quad
\includegraphics[width=.25\textwidth]{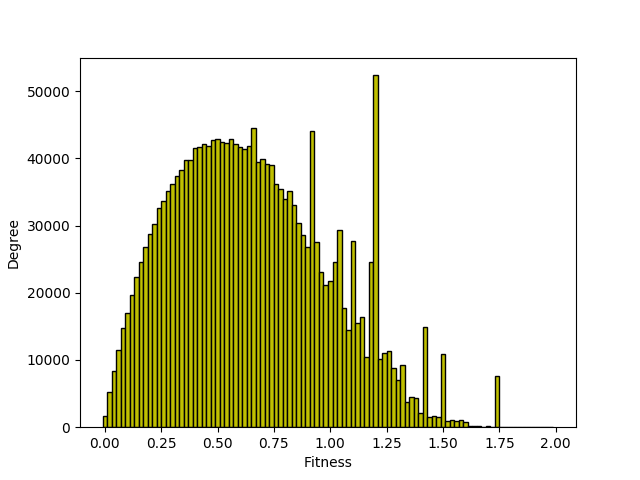}

\caption{Cumulative degree in fitness intervals for Regime~$R2$ with $m=2$ with $\epsilon\sim Beta(1,\,1.9)$. In this case, higher spikes appear as the graph size increases towards the upper endpoint of the fitness distribution. Recall that the classical Bianconi--Barab\'asi model does not condensate for $Beta(1,\,1.9)$-distributed fitness. 
}
\label{pics:tramax_m5_Beta119}
\end{figure}

\begin{figure}[ht!]
\centering
\includegraphics[width=.25\textwidth]{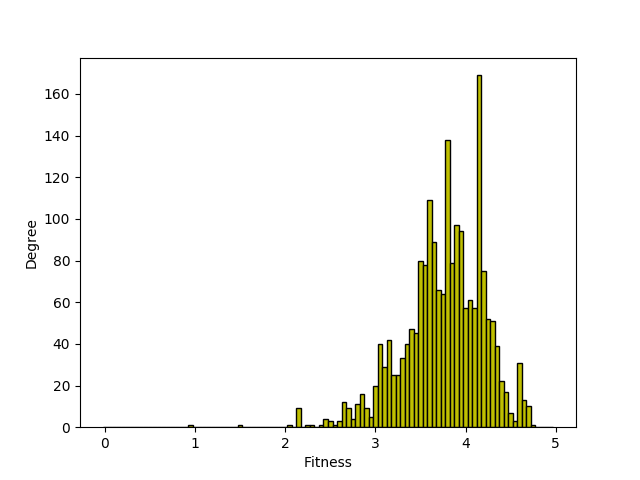}\quad
\includegraphics[width=.25\textwidth]{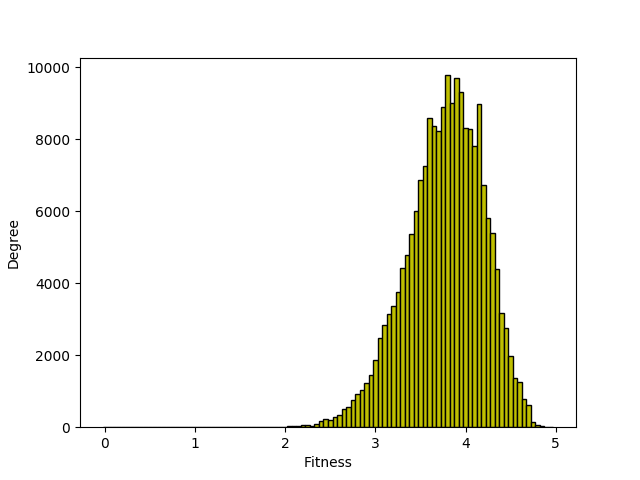}\quad
\includegraphics[width=.25\textwidth]{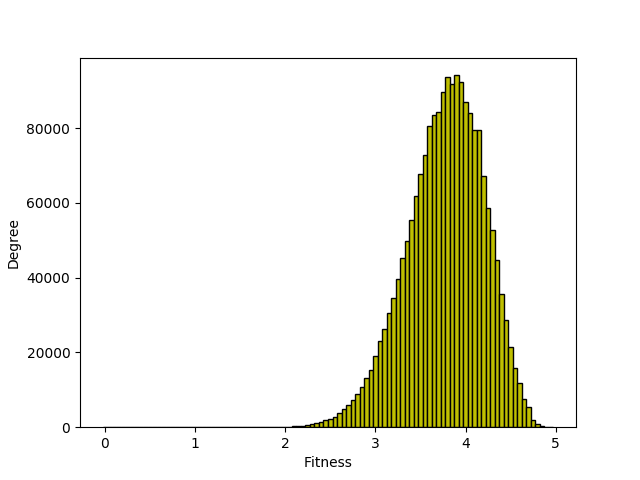}

\caption{Cumulative degree in fitness intervals for Regime~$R2$ with $m=5$, $\epsilon\sim Beta(3,\,1)$. The degree distribution per fitness appears without peaks as the graph increases and no condensates forms.
}
\label{fig:phase_trans_nostro}
\end{figure}
\newpage
\subsection{Condensation in Regime $R3$}
The case in which $m(t)\gg 1$ presents an interesting open problem which is illustrated in the following simulations.

In Figures~\ref{pics:m_log_t}-\ref{pics:m_t} we are plotting the cumulative degree for nodes grouped by fitness with $U(0,\,1)$ increments in the regimes $m=\floor{\log t},\,\floor{\sqrt t},\,t$ respectively. As one can see, the limiting fitness distribution resembles the cumulative fitness distribution for the first two cases, while in the $m=t$ regime a spike appears at $\mu_\eps t$.

As stated at the beginning of Subsection~\ref{subsec:cond}, in this kind of plots condensation is indicated by the presence of spikes around the supremum of the fitness. Looking at Figures~\ref{pics:m_log_t}-\ref{pics:m_sqrt_t} there are no peaks and the landscape of the empirical fitness distribution resembles the cumulative degree grouped by fitness. On the other hand, in Figure~\ref{pics:m_t} the two quantities are different: the cumulative degree by fitness exhibits a spike roughly around $\mu_\eps t$, $t=100000$, and the empirical fitness distribution seems to be uniform in $[0,\,\mu_\eps t].$ Heuristically, a uniform law appears because by summing a linear number of increments the central limit theorem kicks in, so that each node has roughly a Gaussian fitness. More precisely, for most nodes $i$ the fitness is close in law to $\mathcal N(\mu_\eps (t-i),\sigma^2_\eps(t-i))$, where $\sigma^2_\eps$ is the variance of the increments. Therefore, by Gaussian concentration properties around the mean, we see a fitness landscape close to a uniform in $[0,\,\mu_\eps t].$

Based on the above considerations, we expect that the speed at which $m(t)$ grows in time influences the appearance of a condensate. Namely, if $m(t)$ is too slow, condensation cannot be enforced, while a faster $m(t)$ leads to Bose--Einstein condensation. Because of the scaling of the central limit theorem, we conjecture $m(t)=\Theta( t)$ to be the threshold for condensation in $R3$.

This is summarised in the following conjecture:
\begin{conj}
If $m(t)=\Theta( t)$ $R3$ exhibits condensation with the cumulative fitness distribution having an atom at $\mu_\eps t.$
\end{conj}

\begin{figure}[ht!]
\centering
\includegraphics[width=.3\textwidth]{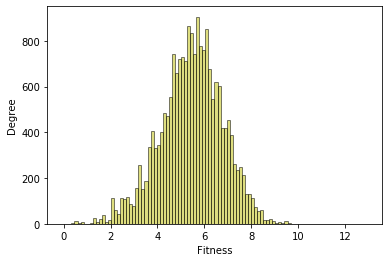}\quad
\includegraphics[width=.3\textwidth]{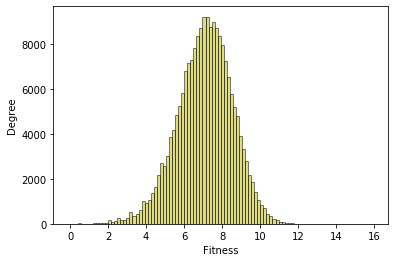}
\includegraphics[width=.3\textwidth]{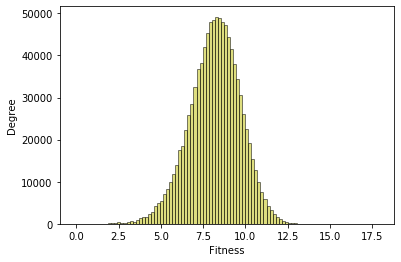}\quad
\medskip

\includegraphics[scale=.35]{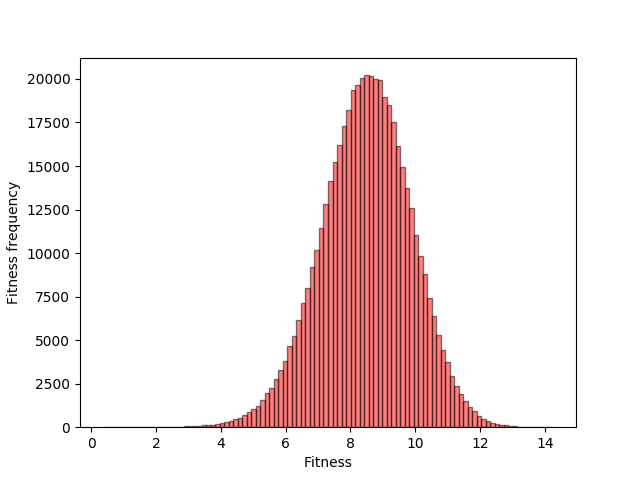}
\caption{Cumulative degree grouped by fitness in Regime~$R3$ with $m=\floor{\log_2{t}}$ with $\epsilon\sim U(0,\,1)$ 
Bottom picture: empirical fitness distribution corresponding to the last cumulative degree plot. 
Notice that the empirical fitness distribution resembles the cumulative degree by fitness. The plots exhibit no spike.
}
\label{pics:m_log_t}
\end{figure}
\begin{figure}[ht!]
\centering
\includegraphics[width=.3\textwidth]{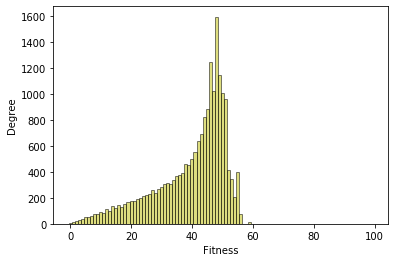}\quad
\includegraphics[width=.3\textwidth]{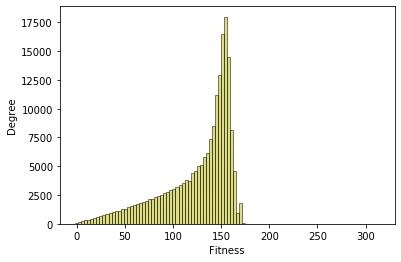}
\includegraphics[width=.3\textwidth]{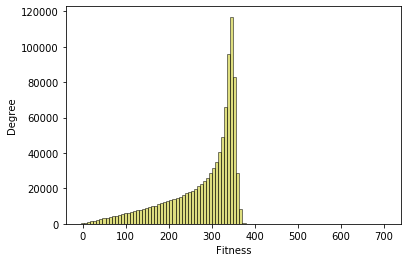}

\medskip

\includegraphics[scale=.35]{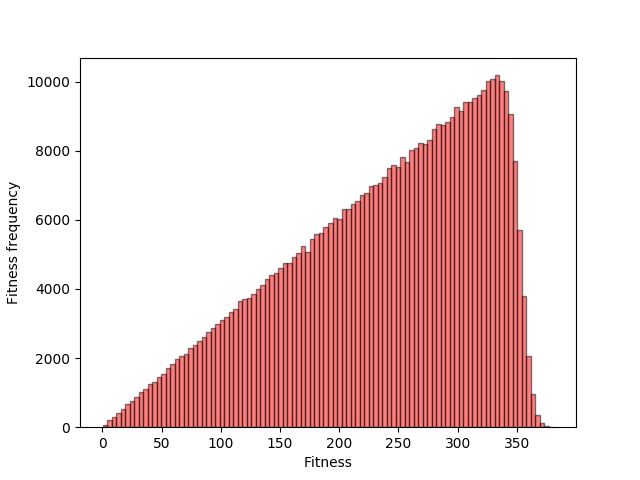}

\caption{Cumulative degree grouped by fitness in Regime~$R3$ with $m=\floor{\sqrt t}$ with $\epsilon\sim U(0,\,1)$ 
Bottom picture: empirical fitness distribution corresponding to the last cumulative degree plot. 
Notice that the empirical fitness distribution resembles the cumulative degree by fitness. The plots exhibit no spike.}
\label{pics:m_sqrt_t}
\end{figure}

\begin{figure}[ht!]
\centering
\includegraphics[width=.3\textwidth]{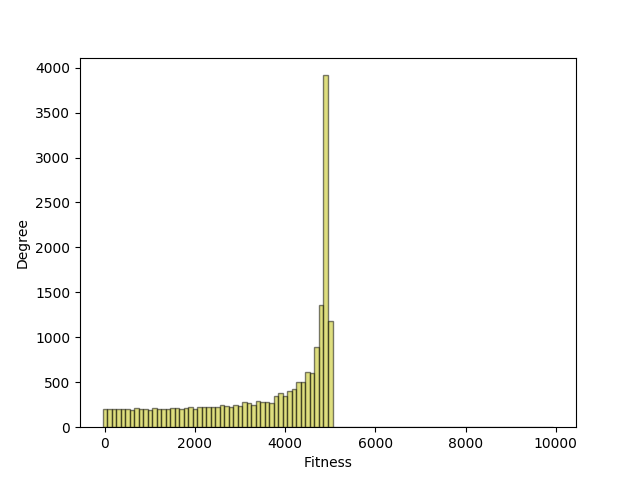}\quad
\includegraphics[width=.3\textwidth]{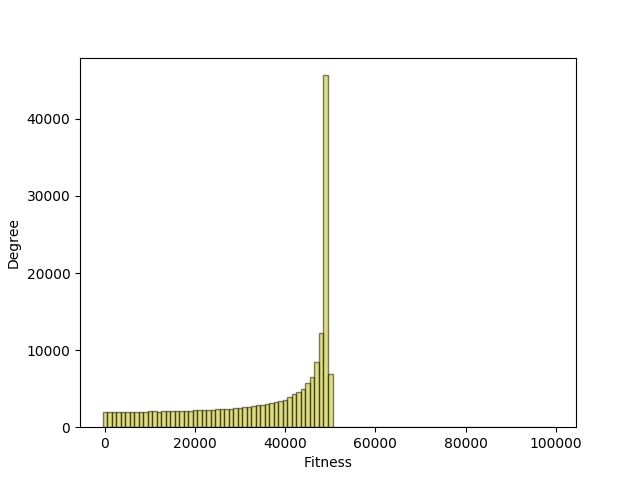}

\medskip

\includegraphics[scale=.3]{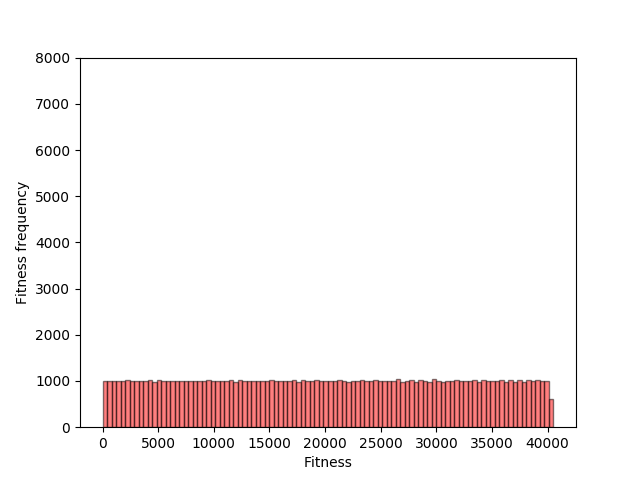}

\caption{Cumulative degree grouped by fitness in Regime~$R3$ with $m=t$ with $\epsilon\sim U(0,\,1)$. Bottom picture: empirical fitness distribution corresponding to the last cumulative degree plot. As expected, the fitness distribution is uniform in $[0,\,\mu_\eps t]$. In particular $\mu_\eps=1/2,$ $t=100000$, and $\mu_\eps t\approx 50000.$ }
\label{pics:m_t}
\end{figure}
We expect our results to be universal regardless of the increment distribution. In order to properly address this topic, we need to identify two fitness families. \cite{mailler2019competing} propose two fitness categories in the context of competing growth processes and dynamical networks. The difference arises essentially in the behavior at the maximal fitness value (regular variation vs. exponential behavior) that implies a different treatment of the two regarding condensation. The families are:
\begin{enumerate}[label=\roman*),ref=\roman*),leftmargin=*]
    \item bounded random variables in the maximum domain of attraction of the Weibull distribution eg. Beta distribution~\parencite[Section~3.3.2]{embrechts2013modelling};
    \item\label{item:Gumbel} bounded random variables in the maximum domain of attraction of the Gumbel distribution eg. those with survival function~\parencite[Section~3.3.3]{embrechts2013modelling}. \begin{equation}\label{eq:eps_G}\overline F(x)=\exp(-x/(1-x)),\quad x\in[0,\,1].\end{equation}
\end{enumerate}

So far we have used in our simulations Beta-distributed increments. In order to better support our claims, we provide a realisation of our model in $R3$ with increments belonging to Class~\ref{item:Gumbel}. Namely, we will use increments distributed as~\eqref{eq:eps_G}. As one can notice, the behavior shown in Figure~\ref{pics:m_t_vonMises} is qualitatively similar to Figure~\ref{pics:m_t}. The intuition behind this is that the central limit theorem works regardless of the initial increment distribution.
\begin{figure}[ht!]
\centering
\includegraphics[width=.3\textwidth]{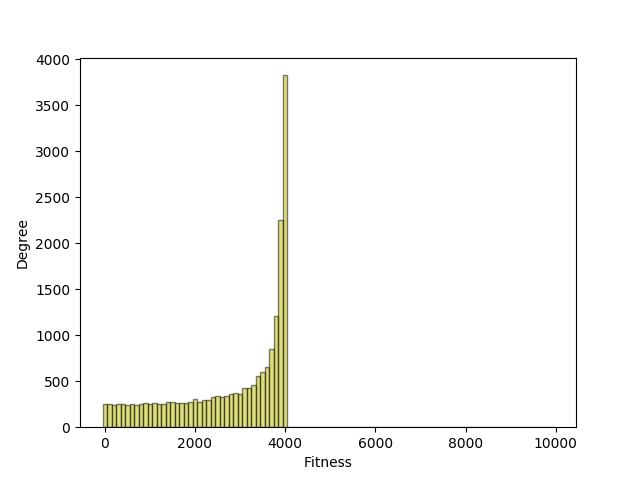}\quad
\includegraphics[width=.3\textwidth]{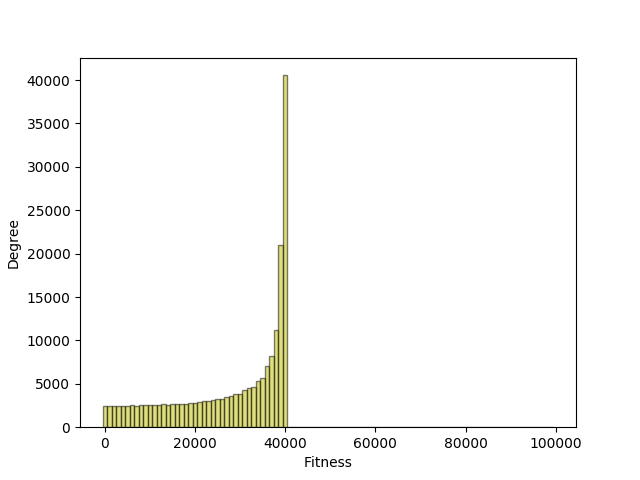}

\medskip

\includegraphics[scale=.3]{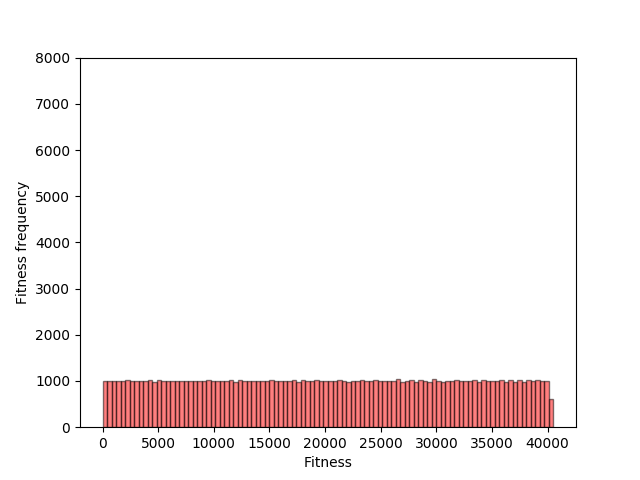}

\caption{Cumulative degree grouped by fitness in Regime~$R3$ with $m=t$ with increments distributed as~\eqref{eq:eps_G}. Bottom picture: empirical fitness distribution corresponding to the last cumulative degree plot. As expected, the fitness distribution is uniform in $[0,\,\mu_\eps t]$. In particular $\mu_\eps=1/2,$ $t=100000$, and $\mu_\eps t\approx 50000.$ }
\label{pics:m_t_vonMises}
\end{figure}
\section{Concluding remarks}\label{sec:concl}
As mentioned in the Introduction most the main tools (urn models, continuous-time branching processes etc) developed to analyze preferential attachment models with fitness are still not able to treat dynamical fitness models. This is why the behavior of these models poses an interesting mathematical challenge. In this paper we started the investigation, both mathematical and empirical, of these models. We believe that a rigorous analysis of dynamical fitness preferential attachment graphs can shed light on the existence of universality classes for random graphs. In particular this can justify the use of the Barab\'asi--Albert and the Bianconi--Barab\'asi models in applications. As an example, since the attachment probability of a graph is hard to estimate, knowing that the benchmark models are robust under bounded fluctuations of the fitness makes them suitable to fit observed networks. 

Finally, a rigorous study of Regime~$R3$ is advocated. The reason for this is that it creates a new universal model where the phase transition seems not to depend anymore on the fitness but rather on the speed of the fitness growth. 
\clearpage
\printbibliography

\end{document}